\documentclass[12pt,twoside,reqno]{amsart}
\usepackage{amsxtra,amscd}
\usepackage{graphicx}
\usepackage{amsmath}
\usepackage{amsfonts}
\usepackage{amssymb}

\newtheorem{theorem}{Theorem}[section]
\newtheorem{lemma}[theorem]{Lemma}
\theoremstyle{definition}
\newtheorem{definition}[theorem]{Definition}

\newtheorem{remark}[theorem]{Remark}

\numberwithin{equation}{section}

\begin{document}
\title{on the \'{e}tale fundamental groups of arithmetic schemes}
\author{Feng-Wen An}
\address{School of Mathematics and Statistics, Wuhan University, Wuhan,
Hubei 430072, People's Republic of China}
\email{fwan@amss.ac.cn}
\subjclass[2000]{Primary 14F35; Secondary 11G35}
\keywords{arithmetic scheme, automorphism group, \'{e}tale fundamental
group, Galois group, quasi-galois}

\begin{abstract}
In this paper we will give the computation of the \'{e}tale fundamental
group of an arithmetic scheme.
\end{abstract}

\maketitle

\begin{center}
{\tiny {Contents} }
\end{center}

{\tiny \qquad {Introduction} }

{\tiny \qquad {1. Statement of The Main Theorem} }

{\tiny \qquad {2. Preliminaries}}

{\tiny \qquad {3. Construction for the Model}}

{\tiny \qquad {4. Proof of The Main Theorem} }

{\tiny \qquad {References}}

\section*{Introduction}

It has been seen that
 \'{e}tale fundamental groups of arithmetic varieties encode the whole of
the information of the maximal abelian class fields of the number fields (e.g., \cite{Bloch,Kato,Schmidt,Raskind,Saito,VS1,w1,w2}). Here one uses the related data of arithmetic varieties to describe class
fields for a given field.

In particular, one uses the
varieties $X/Y$  as a geometric model for describing such
extensions $E/F$ that the Galois group $Gal\left( E/F\right) $ is isomorphic
to the automorphism group ${Aut}\left( X/Y\right) $ (e.g., \cite%
{SGA1,Raskind,VS1,SV2}).

In this paper we will give the computation of the \'{e}tale fundamental
group of an arithmetic scheme. Here the key point to overcome is such as the following:

For any arithmetic variety $X$, there is a directed set $X_{qc}[\Omega]$ made up of quasi-galois closed
covers over $X$ which have reduced affine coverings with values in  $\Omega^{un}$, satisfying the properties
\begin{itemize}
\item $X_{qc}[\Omega]$ is a subset of $X_{et}[\Omega]$;

\item $X_{qc}[\Omega]$ and $X_{et}[\Omega]$ are cofinal directed sets.
\end{itemize}
Here $\Omega$ is a fixed algebraic closure of the function field $k(X)$, $\Omega^{un}$ is the subfield consisting of finite unramified extensions over $k(X)$, and
$X_{et}[\Omega]$ is the set of finite \'{e}tale Galois covers over $X$ preserving the given geometric point $s$ of $X$ over $\Omega$.

\subsection*{Acknowledgment}

The author would like to express his sincere gratitude to Professor Li
Banghe for his advice and instructions on algebraic geometry and topology.

\section{Statement of the Main Theorem}

\subsection{Notation and Definition}

By an \textbf{arithmetic variety} $X$ in the paper, we will understand an
integral scheme surjectively over $Spec\left( \mathbb{Z}\right) $ of finite type.

Let $\pi _{1}^{et}\left( X,s\right) $ denote the \'{e}%
tale fundamental group of $X$ for a given geometric point $s$ of $X.$

As usual, let $k(Z)$ ($\triangleq \mathcal{O}_{Z,\xi}$) denote the function field of an integral scheme $Z$ (with generic point $\xi$); let $Gal(K_{2}/K_{1})$ denote the Galois group of the field extension $K_{2}/K_{1}$.

\begin{definition}
Let $K_{1}\subseteq K_{2}$ be two finitely generated function fields over a
number field $K.$

$\left( i\right) $ $K_{2}$ is said to be a
finite \textbf{unramified Galois} extension of $K_{1}$ if there are two arithmetic varieties $X_{1}$
and $X_{2}$ and a surjective morphism $f:X_{2}\rightarrow X_{1}$ such that
\begin{itemize}
\item $k\left( X_{1}\right) =K_{1},k\left( X_{2}\right) =K_{2}$;

\item $X_{2}$ is a
finite \'{e}tale Galois cover of $X_{1}$ by $f$.
\end{itemize}

$\left( ii\right) $ $K_{2}$ is said to be a finite \textbf{%
unramified} extension of $K_{1}$ if there is a finitely generated function
field $K_{3}$ over $K$ such that $K_{2}$ is contained in $K_{3}$ and $K_{3}$
is a finite unramified Galois extension of $K_{1}.$
\end{definition}

Now let $L$ be a function field over a number field $K$ (not necessarily finitely generated). Set
\begin{itemize}
\item $L^{al}\triangleq $
an algebraical closure of $L$;

\item $L^{un}\triangleq$
the union of all the finite unramified subextensions (in $L^{al}$) over $L.$
\end{itemize}

\begin{remark}
Let $L$ be a function field over a number field $K$.

$(i)$ If $L/K$ is algebraic, $L^{un}$ is exactly the case in algebraic number theory.

$(ii)$ By the main result in \cite{An5}, it can be seen that $L^{un}/L$ and $L^{al}/L^{un}$ are both Galois extensions.
\end{remark}

\subsection{Statement of the Main Theorem}

Here is the main result of the paper.

\begin{theorem}
Fixed any arithmetic variety $X$. Then there exists an isomorphism
\begin{equation*}
\pi _{1}^{et}\left( X,s\right) \cong Gal\left( {k(X)}^{un}/k\left( X\right)
\right)
\end{equation*}
between groups
for any geometric point $s$ of $X$ over an algebraic closure of the function
field $k\left( X\right) .$
\end{theorem}

We will prove the main theorem in \S 4 after the preparations are made in \S\S 2-3.

\section{Preliminaries}

For convenience, let us fix notation and definitions and give (or recall)
preliminary facts.

\subsection{Notation}

Fixed an integral domain $D$. In the paper, we let $Fr(D)$ denote the field of fractions on $D$.

If $D$ be a subring of a field $\Omega $, the field $Fr(D)$ will always
assumed to be contained in $\Omega $.

Let $E$ be an extension of a field $F$ (not necessarily algebraic). $E$ is said to be \textbf{Galois} over $F$ if $F$ is the fixed subfield of the Galois group $Gal(E/F)$.

By an \textbf{integral variety}, we will understand an integral scheme surjectively over $Spec(\mathbb{Z})$ (not necessarily of finite type).

\subsection{Affine Covering with Values}

Fixed a scheme $X$. As usual, an affine covering of the scheme $X$ is a
family $$\mathcal{C}_{X}=\{(U_{\alpha },\phi _{\alpha };A_{\alpha
})\}_{\alpha \in \Delta }$$ such that for each $\alpha \in \Delta $, $\phi
_{\alpha }$ is an isomorphism from an open set $U_{\alpha }$ of $X$ onto the
spectrum $Spec{A_{\alpha }}$ of a commutative ring $A_{\alpha }$. Each $%
(U_{\alpha },\phi _{\alpha };A_{\alpha })\in \mathcal{C}_{X}$ is called a
\textbf{local chart}.

An affine covering $\mathcal{C}_{X}$ of $X$ is said to
be \textbf{reduced} if $U_{\alpha }\neq U_{\beta }$ holds for any $\alpha
\neq \beta $ in $\Delta $.

Let $\mathfrak{Comm}$ be the category of commutative rings with identity.
Fixed a subcategory $\mathfrak{Comm}_{0}$ of $\mathfrak{Comm}$. An affine
covering $\{(U_{\alpha },\phi _{\alpha };A_{\alpha })\}_{\alpha \in \Delta }$
of $X$ is said to be \textbf{with values} in $\mathfrak{Comm}_{0}$ if
 for each $\alpha \in \Delta $ there are $\mathcal{O}_{X}(U_{\alpha})=A_{\alpha}$ and $U_{\alpha}=Spec(A_{\alpha})$, where
 $A_{\alpha }$ is a ring contained in $\mathfrak{Comm}_{0}$.
 
In
particular, let $\Omega $ be a field and let $\mathfrak{Comm}(\Omega )$ be
the category consisting of the subrings of $\Omega $ and their isomorphisms.
An affine covering $\mathcal{C}_{X}$ of $X$ with values in $\mathfrak{Comm}%
(\Omega )$ is said to be \textbf{with values in the field $\Omega $}.

\subsection{Definition of Quasi-Galois Closed}

Let $X$ and $Y$ be integral varieties and let $f:X\rightarrow Y$ be a
surjective morphism. Denote by $Aut\left( X/Y\right) $ the group of
automorphisms of $X$ over $Z$.

By a \textbf{conjugate} $Z$ of $X$ over $Y$, we understand an integral
variety $Z$ that is isomorphic to $X$ over $Y$.

Let $\mathcal{O}_{X}$ and $\mathcal{O}^{\prime}_{X}$ be two structure sheaves on the underlying space of an integral scheme $X$. The two integral schemes $(X,\mathcal{O}_{X})$ and $(X, \mathcal{O}^{\prime}_{X})$ are said to be \textbf{essentially equal} provided that for any open set $U$ in $X$, we have
 $$U \text{ is affine open in }(X,\mathcal{O}_{X}) \Longleftrightarrow \text{ so is }U \text{ in }(X,\mathcal{O}^{\prime}_{X})$$ and in such a case,  $D_{1}=D_{2}$ holds or  there is $Fr(D_{1})=Fr(D_{2})$ such that for any nonzero $x\in Fr(D_{1})$, either $$x\in D_{1}\bigcap D_{2}$$ or $$x\in D_{1}\setminus D_{2} \Longleftrightarrow x^{-1}\in D_{2}\setminus D_{1}$$ holds, where $D_{1}=\mathcal{O}_{X} (U)$ and $D_{2}=\mathcal{O}^{\prime}_{X} (U)$.

 Two schemes $(X,\mathcal{O}_{X})$ and $(Z,\mathcal{O}_{Z})$ are said to be \textbf{essentially equal} if the underlying spaces of $X$ and $Z$ are equal and the schemes $(X,\mathcal{O}_{X})$ and $(X,\mathcal{O}_{Z})$ are essentially equal.
 
\begin{definition}
$X$ is said to be \textbf{quasi-galois closed} over $Y$ by $f$ if there is an algebraically closed field $\Omega$
and a reduced affine covering $\mathcal{C}_{X}$ of $X$ with values in $
\Omega $ such that for any conjugate $Z$ of
$X$ over $Y$ the two conditions are satisfied:
\begin{itemize}
\item $(X,\mathcal{O}_{X})$ and $(Z,\mathcal{O}_{Z})$ are essentially equal if $Z$ has a reduced
affine covering with values in $\Omega$.

\item $\mathcal{C}_{Z}\subseteq \mathcal{C}_{X}$ holds if $\mathcal{C}_{Z}$
is a reduced affine covering of $Z$ with values in $\Omega $.
\end{itemize}
\end{definition}

\begin{remark}
$\left( i\right) $ In fact, here $\Omega $ is  an algebraical closure
of the function field $k\left( X\right) $ (See \cite{An2}).

$\left( ii\right) $ A quasi-galois closed variety $X$ has only one conjugate
(over $Y$) in an algebraic closure $\Omega $ of the function field $k\left( X\right) $. That is, let $Z$ and $Z^{\prime }$ be conjugates of $X.$
Then we must have $Z=Z^{\prime }$ if $Z$ and $Z^{\prime }$ both have reduced
affine coverings with values in $\Omega .$
\end{remark}

\subsection{Existence of Quasi-Galois Closed}

By the lemma below we can take an arithmetic variety as a scheme that has a reduced affine covering in a field.

\begin{lemma}
Let $Y$ be an arithmetic variety and let $\Omega $ be an algebraic closure
of the function field $k\left( Y\right) $. Then there is an arithmetic
variety $Z$ satisfying the conditions:

\begin{itemize}
\item $k\left( Y\right) =k\left( Z\right) \subseteq \Omega ;$

\item $Y\cong Z$ are isomorphic;

\item $Z$ has a reduced affine covering with values in $\Omega .$
\end{itemize}
\end{lemma}

\begin{proof}
It is immediate from definition for affine covering with values in a given
field.
\end{proof}

There is the lemma below for the existence of quasi-galois closed.

\begin{lemma}
\emph{(See \cite{An3})} Let $Y$ be an arithmetic variety with $K=k\left( Y\right) $ and let $L$ be a finitely generated extension of $K$ such that $L$ is Galois
over $K$.
Then there exists an arithmetic variety $X$ and a surjective morphism $%
f:X\rightarrow Y$ of finite type such that

\begin{itemize}
\item $L=k\left( X\right) $;

\item $f$ is affine;

\item $X$ is a quasi-galois closed over $Y$ by $f$.
\end{itemize}
\end{lemma}

\subsection{Properties of Quasi-Galois Closed}

Quasi-galois closed schemes have the following properties.

\begin{lemma}
\emph{(See \cite{An2})} Let $X$ and $Y$ be two arithmetic varieties such
that $X$ is quasi-galois closed over $Y$ by a surjective morphism $f$ of
finite type. Then there are the following statements.

\begin{itemize}
\item $f$ is affine.

\item $k\left( X\right) $ is canonically Galois over $k(Y)$.

\item There is a group isomorphism
\begin{equation*}
{Aut}\left( X/Y\right) \cong Gal(k\left( X\right) /k(Y)).
\end{equation*}

\item Let $\dim X=\dim Y$. Then $X$ is a pseudo-galois
cover over $Y$ in the sense of Suslin-Voevodsky.
\end{itemize}
\end{lemma}

\begin{lemma}
\emph{(See \cite{An5})} Let $f:X\rightarrow Y$ be a surjective morphism of integral varieties.
Suppose that $X$ is quasi-galois closed over $Y$ by $f$ and $k\left( X\right) $
is canonically Galois over $k\left( Y\right) .$ Then $f$ is
affine and there is a group isomorphism
\begin{equation*}
Aut\left( X/Y\right) \cong Gal\left( k\left( X\right) /k\left( Y\right)
\right) .
\end{equation*}
\end{lemma}

\begin{remark}
Let $X$ and $Y$ be integral varieties such that $X$ is quasi-galois closed
over $Y$ by a surjective morphism $f$. Then there is a natural
isomorphism
\begin{equation*}
\mathcal{O}_{Y}\cong f _{\ast }(\mathcal{O}_{X})^{G}
\end{equation*}
between sheaves,
where $(\mathcal{O}_{X})^{G}(U)$ denotes the invariant
subring of $\mathcal{O}_{X}(U)$ under the natural action of $G\triangleq {Aut}\left(
X/Y\right) $ for any open subset $U$ of $X$.
\end{remark}

\subsection{Criteria of Quasi-Galois Closed}

Let $K$ be an extension of a field $k$ (not necessarily algebraic).

\begin{definition}
(See \cite{An2})
 $K$ is said to be \textbf{quasi-galois} over $k$ if every
irreducible polynomial $f(X)\in F[X]$ that has a root in $K$ factors
completely in $K\left[ X\right] $ into linear factors for any subfield $F$ with $k\subseteq F\subseteq K$.
\end{definition}

Let $D\subseteq D_{1}\cap D_{2}$ be three integral domains. $D_{1}$
is said to be \textbf{quasi-galois} over $D$ if $Fr\left(
D_{1}\right) $ is quasi-galois over $Fr\left( D\right) $.

\begin{definition}
$D_{1}$ is said to be a \textbf{conjugation} of $D_{2}$ over $D$ if
there is an $F-$isomorphism $\tau:Fr(D_{1})\rightarrow
Fr(D_{2})$ such that
$
\tau(D_{1})=D_{2}$,
where $F\triangleq k(\Delta)$, $k\triangleq Fr(D)$, $\Delta$ is a transcendental basis of the field $Fr(D_{1})$ over $k$, and $F$ is contained in $Fr(D_{1})\cap Fr(D_{2})$.
\end{definition}

Now let $X$ and $Y$ be two integral varieties and let $f:X\rightarrow Y$ be a
surjective morphism. Put $\Omega =k(X)^{al}$.

\begin{definition}
A reduced affine covering $\mathcal{C}_{X}$ of $X$ with values in $\Omega $
is said to be \textbf{quasi-galois closed} over $Y$ by $f$ if the
below condition is satisfied:

There exists a local chart $(U_{\alpha }^{\prime },\phi _{\alpha }^{\prime
};A_{\alpha }^{\prime })\in \mathcal{C}_{X}$ such that $U_{\alpha }^{\prime
}\subseteq \varphi^{-1}(V_{\alpha})$ for any $(U_{\alpha },\phi _{\alpha
};A_{\alpha })\in \mathcal{C}_{X}$, for any affine open set $V_{\alpha}$ in $%
Y$ with $U_{\alpha }\subseteq f^{-1}(V_{\alpha})$, and for any
conjugate $A_{\alpha }^{\prime }$ of $A_{\alpha }$ over $B_{\alpha}$, where $%
B_{\alpha}$ is the canonical image of $\mathcal{O}_{ Y}(V_{\alpha})$ in the
function field $k(X)$ via $f$.
\end{definition}

\begin{definition}
(See \cite{An3})
An affine covering $\{(U_{\alpha },\phi _{\alpha };A_{\alpha })\}_{\alpha
\in \Delta }$ of $X$ is said to be an \textbf{affine
patching} of $X$ if the map $\phi _{\alpha }$ is the
identity map on $U_{\alpha }=SpecA_{\alpha }$ for each $\alpha \in \Delta .$
\end{definition}

Evidently, an affine patching is reduced.

\begin{lemma}
\emph{(See \cite{An3})} Let $X$ and $Y$ be arithmetic varieties and let $f$ be of finite type. Suppose $k(Y) \subseteq \Omega$. Then $X$ is quasi-galois closed over $Y$ if there is a unique
maximal affine patching $\mathcal{C}_{X}$ of $X$ with values in $\Omega $
such that

\begin{itemize}
\item either $\mathcal{C}_{X}$ is quasi-galois closed over $Y$ by $f$,

\item or $A_{\alpha }$ has only one conjugate over $B_{\alpha }$ for any $%
(U_{\alpha },\phi _{\alpha };A_{\alpha })\in \mathcal{C}_{X}$ and for any
affine open set $V_{\alpha }$ in $Y$ with $U_{\alpha }\subseteq f
^{-1}(V_{\alpha })$, where $B_{\alpha }$ is the canonical image of $\mathcal{%
O}_{Y}(V_{\alpha })$ in the function field $k(Y)$.
\end{itemize}
\end{lemma}

\begin{lemma}
Let $f :X\rightarrow Y$ be a surjective morphism of integral varieties.
Suppose $k(Y) \subseteq \Omega$. Then $X$ is
quasi-galois closed over $Y$ by $f$ if  there is a unique maximal affine
patching $\mathcal{C}_{X}$ of $X$ with values in $\Omega $ such that $%
\mathcal{C}_{X}$ is quasi-galois closed over $Y$.
\end{lemma}

\begin{proof}
Assume that there is a unique maximal affine
patching $\mathcal{C}_{X}$ of $X$ with values in $\Omega $ such that $\mathcal{C}_{X}$ is quasi-galois closed over $Y$.
Let $Z$ be a conjugate of $X$ over $Y$ such that $Z$ has a reduced affine covering $\mathcal{C}_{Z}$ with values in $\Omega$.
It suffices to prove that the two schemes $X$ and $Z$ are equal.

In fact, let $\sigma: X \rightarrow Z$ be an isomorphism over $Y$. Fixed any local chart $(U,\phi_{U};A_{U})\in \mathcal{C}_{X}$. We have $\mathcal{O}_{X}(U)=A_{U}$ and $U=Spec(A_{U})$.

Put $V=\sigma (U)$ and $B_{V}=\mathcal{O}_{Z}(V)$. As $A_{U}=\sigma^{\sharp}(B_{V})$ holds, it is seen that $B_{V}$ is a conjugation of $A_{U}$. Then we have a local chart $(U^{\prime},\phi_{U^{\prime}};A_{U^{\prime}})\in \mathcal{C}_{X}$ such that $A_{U^{\prime}}=B_{V}$ since $\mathcal{C}_{X}$ is quasi-galois closed over $Y$. This proves $V\subseteq X$ and then $Z \subseteq X$.

 We must have $X=Z$ as topological spaces since via $\sigma$ it is seen that $Z$ is closed and open in $X$. It follows that $\mathcal{C}_{X}$ is also an affine
patching of $Z$.
Then $\mathcal{O}_{X}(U)=\mathcal{O}_{Z}(U)$ holds for any affine open set $U$ in $X$. Furthermore, let $U_{1}\supseteq U_{2}$ be two affine open sets in $X$. The maps $$r_{X}^{U_{1},U_{2}},r_{Z}^{U_{1},U_{2}}:\mathcal{O}_{X}(U_{1})\rightarrow \mathcal{O}_{X}(U_{2})$$ are equal, where $r_{X}^{U_{1},U_{2}}$ and $r_{Z}^{U_{1},U_{2}}$ are the restrictions of $\mathcal{O}_{X}$ and $\mathcal{O}_{X}$, respectively.

Hence, the stalks $\mathcal{O}_{X,w}$ and $\mathcal{O}_{Z,w}$ coincide with each other at every point $w$ in the underlying space $X$.

It is seen that $\mathcal{O}_{X}$ can be identified with $\mathcal{O}_{Z}$, that is, $(X,\mathcal{O}_{X})=(Z,\mathcal{O}_{Z})$. This proves that  $X$ is
quasi-galois closed over $Y$ by $f$.
\end{proof}

\section{Construction for the Model}

In this section we will construct an integral scheme $X_{\infty }^{un}$ over
$Spec\left( \mathbb{Z}\right) $ such that $k\left( X_{\infty }^{un}\right)
=k\left( X\right) ^{un}.$

Let $Y$ be an arithmetic variety. For brevity, put $K=k\left( Y\right) ,$ $%
\Omega =k\left( Y\right) ^{al},$ and $L=K^{un}\subseteq \Omega .$

By \emph{Lemma 2.1}, without loss of generality, the scheme $Y$ is assumed
to have a reduced affine covering $\mathcal{C}_{Y}$ with values in $\Omega .$
We choose $\mathcal{C}_{Y}$ to be maximal (in the sense of set inclusion).

In the following we will proceed in several steps to construct a scheme $X$
and a surjective morphism $f:X\rightarrow Y$ such that $k\left( X\right) =L$.

\emph{\textbf{Step 1.}} Fixed a set $\Delta $ of generators of the field $L$
over $K$ with $\Delta \subseteq L\setminus K.$ Put $G=Gal\left( L/K\right) .$

\emph{\textbf{Step 2.}} Take any local chart $\left( V,\psi
_{V},B_{V}\right) \in \mathcal{C}_{Y}.$ As $V$ is affine open in $Y,$ we
have
$
Fr\left( B_{V}\right) =K$ and $\mathcal{O}_{Y}\left( V\right)
=B_{V}\subseteq \Omega .
$

Define $A_{V}=B_{V}\left[ \Delta _{V}\right] ,$ i.e., the subring of $L$
generated over $B_{V}$ by the set
$
\Delta _{V}=\{\sigma \left( x\right) \in L:\sigma \in G,x\in \Delta \}.
$
Let $i_{V}:B_{V}\rightarrow A_{V}$ be the inclusion.

\emph{\textbf{Step 3.}} Define
$$
\Sigma =\coprod\limits_{\left( V,\psi _{V},B_{V}\right) \in \mathcal{C}%
_{Y}}Spec\left( A_{V}\right)
$$
to be the disjoint union. Let $\pi _{Y}:\Sigma \rightarrow Y$ be the
projection induced by the inclusions $i_{V}$.

Then $\Sigma $ is a topological space, where the topology $\tau _{\Sigma }$
on $\Sigma $ is naturally determined by the Zariski topologies on all $%
Spec\left( A_{V}\right) .$

\emph{\textbf{Step 4.}} Given an equivalence relation $R_{\Sigma }$ in $%
\Sigma $ in such a manner:

For any $x_{1},x_{2}\in \Sigma $, we say $x_{1}\sim x_{2}$ if and only if $%
j_{x_{1}}=j_{x_{2}}$ holds in $L$. Here, $j_{x}$ denotes the corresponding
prime ideal of $A_{V}$ to a point $x\in Spec\left( A_{V}\right) $ (see \cite%
{EGA}).

Let $X=\Sigma /\sim $ and let $\pi _{X}:\Sigma \rightarrow X$ be the
projection.

It is seen that $X$ is a topological space as a quotient of $\Sigma .$

\emph{\textbf{Step 5.}} Given a map $f:X\rightarrow Y$ of spaces by
$
\pi _{X}\left( z\right) \longmapsto \pi _{Y}\left( z\right)
$
for each $z\in \Sigma $.

\emph{\textbf{Step 6.}} Define
\begin{equation*}
\mathcal{C}_{X}=\{\left( U_{V},\varphi _{V},A_{V}\right) \}_{\left( V,\psi
_{V},B_{V}\right) \in \mathcal{C}_{Y}}
\end{equation*}%
where $U_{V}=\pi _{Y}^{-1}\left( V\right) $ and $\varphi
_{V}:U_{V}\rightarrow Spec(A_{V})$ is the identity map for each $\left(
V,\psi _{V},B_{V}\right) \in \mathcal{C}_{Y}$.

Define $X_{\infty }^{un}$ to be the scheme $X$ obtained by gluing the affine
schemes $Spec\left( A_{V}\right) $ for all $\left( U_{V},\varphi
_{V},A_{V}\right) \in \mathcal{C}_{X}$ with respect to the equivalence
relation $R_{\Sigma }$ (see \cite{EGA,Hrtsh}).

Then $X_{\infty }^{un}$ is the desired scheme and $f_{\infty }=f:X_{\infty
}^{un}\rightarrow Y$ is the desired morphism of schemes. This completes the
construction.

\begin{remark}
$\mathcal{C}_{X}$ is a reduced affine covering of the scheme $X$ with values
in $\Omega .$ In particular, $\mathcal{C}_{X}$ is maximal (by set inclusion).
\end{remark}

\section{Proof of the Main Theorem}

Now we can give the proof of the Main Theorem of the paper.

\begin{proof}
\textbf{(Proof of Theorem 1.3)} Fixed an algebraic closure $\Omega $\ of the
function field $k\left( X\right) .$ Without loss of generality, assume that $%
X$ has a reduced a reduced affine covering $\mathcal{C}_{X}$ with values in $%
\Omega $ by \emph{Lemma 2.3}.

Let $\Delta \subseteq k(X)^{un}\setminus
k\left( X\right) $ be a set of generators of the field $k(X)^{un}$ over $%
k(X) $. Define $I=\{\text{finite subsets of }\Delta \}.$ Put $G=Gal\left( k(X)^{un}/k\left( X\right)
\right) .$

In the following we will proceed in several steps to demonstrate the theorem.

\emph{Step 1.} By the construction for $\Delta $ and $G$ in \S 3, we obtain
an integral scheme $X_{\infty }^{un}$ over $Spec\left( \mathbb{Z}\right) $
and a surjective morphism $f_{\infty }:X_{\infty }^{un}\rightarrow X$
satisfying the properties:

\begin{itemize}
\item $k\left( X_{\infty }^{un}\right) =k\left( X\right) ^{un}$;

\item $f_{\infty }$ is affine;

\item $k\left( X_{\infty }^{un}\right) $ is Galois over $k\left( X\right) ;$

\item $X_{\infty }^{un}/X$ is quasi-galois closed by $f_{\infty }$.
\end{itemize}

In deed, by \emph{Steps 2} and \emph{4} below it is seen that $k\left( X_{\infty }^{un}\right) $ is Galois over $k\left( X\right) $; then by \emph{Lemma 2.13} it is seen that $X_{\infty }^{un}/X$ is quasi-galois closed by $f_{\infty }$.

\emph{Step 2.} Fixed any ${\alpha }$ in $I$. Repeat the construction for $%
\alpha $ and $G$ in \S 3; then we obtain an arithmetic variety $X_{\alpha
}^{un}$ and a surjective morphism $f_{\alpha }:X_{\alpha }^{un}\rightarrow X$
satisfying the properties:

\begin{itemize}
\item $k\left( X_{\alpha }^{un}\right) \subseteq k\left( X\right) ^{un}$;

\item $f_{\alpha }$ is affine;

\item $k\left( X_{\alpha }^{un}\right) $ is Galois over $k\left( X\right) ;$

\item $X_{\alpha }^{un}/X$ is quasi-galois closed by $f_{\alpha }$.
\end{itemize}

In deed, it is immediate from \emph{Lemmas 2.5} and \emph{2.12}.

\emph{Step 3.} Fixed any ${\alpha },{\beta }$ in $I.$

Assume $\alpha \subseteq \beta $. We have arithmetic varieties $X_{{\alpha }%
}^{un}$ and $X_{{\beta }}^{un}$ and a surjective morphism $$f_{\alpha
}^{\beta }:X_{\beta }^{un}\rightarrow X_{\alpha }^{un}$$ of finite type. Here
$f_{\alpha }^{\beta }$ is obtained in an manner similar to $f_{\infty }$ as
we have done in \S 3.

Then we have $$f_{\beta }=f_{\alpha }\circ f_{\alpha }^{\beta }.$$ It is seen that $X_{%
{\alpha }}^{un}/X$, $X_{{\beta }}^{un}/X$, and $X_{{\beta }}^{un}/X_{{\alpha
}}^{un}$ are all quasi-galois closed.

In particular, there is a $\gamma $ in $I$ such that $\gamma \supseteq
\alpha $ and $\gamma \supseteq \beta .$ It is seen that $X_{{\gamma }%
}^{un}/X_{{\beta }}^{un}$ and $X_{{\gamma }}^{un}/X_{{\alpha }}^{un}$ are
both quasi-galois closed.

\emph{Step 4.} Take any $\alpha ,\beta $ in $I$. We say $\alpha \leq \beta $
if and only if $\alpha \subseteq \beta .$ It is seen that $I$ is a partially
ordered set.

 Then $\{k\left( X_{{\alpha }}^{un}\right) ;i_{\alpha }^{\beta
}\}_{\alpha \in I}$ is a direct system of groups, where each $$i_{\alpha
}^{\beta }:k\left( X_{{\alpha }}^{un}\right) \rightarrow k\left( X_{{\beta }%
}^{un}\right) $$ is a homomorphism of fields determined by the morphism $%
f_{\alpha }^{\beta }$.

It is clear that $$k\left( X_{\infty }^{un}\right) =\lim_{\rightarrow
_{\alpha \in I}}k\left( X_{{\alpha }}^{un}\right) $$ holds as fields.
For the Galois groups, we have
$$Gal\left( k\left( X_{\infty }^{un}\right) /k\left( X\right) \right) \cong
\lim_{\leftarrow _{\alpha \in I}}Gal\left( k\left( X_{\alpha }^{un}\right)
/k\left( X\right) \right) .$$

\emph{Step 5.} Put $$X_{qc}\left[ \Omega \right] =\{X_{\alpha }:\alpha \in
I\}.$$

Then $X_{qc}\left[ \Omega \right] $ is a directed set, where for any $X_{\alpha },X_{\beta }\in X_{qc}\left[ \Omega \right] ,$ we say $X_{\alpha
}\leq X_{\beta }$ if and only if $X_{{\beta }}^{un}$ is quasi-galois closed
over $X_{{\alpha }}^{un}$.

Fixed a geometric point $s$ of $X$ over $\Omega .$ Put
\begin{equation*}
\begin{array}{l}
X_{et}\left[ \Omega %
\right] \\

=\{Z\text{ is a finite \'{e}tale Galois cover of } X \text{ with a geometric point
over }s\}.
\end{array}
\end{equation*}

 Then $X_{et}\left[ \Omega \right] $ is a directed set, where for
any $X_{1},X_{2}\in X_{et}\left[ \Omega \right] ,$ we say $X_{1}\leq X_{2}$
if and only if $X_{2}$ is a finite \'{e}tale Galois cover over $X_{1}.$

Let $X_{\alpha },X_{\beta }\in X_{qc}\left[ \Omega \right] .$ Suppose that $%
X_{\beta }/X_{\alpha }$ is quasi-galois closed. Then $%
X_{\beta }$ must be a finite \'{e}tale Galois cover over $X_{\alpha }.$
Hence,  $X_{qc}\left[ \Omega \right] $ is a directed subset of $X_{et}\left[
\Omega \right] .$

\emph{Step 6.} Take any $Z\in X_{et}\left[ \Omega \right] .$ We have $%
k\left( Z\right) \subseteq k\left( X\right) ^{un}.$ It is seen that $k\left(
Z\right) $ is a finite unramified extension over $k\left( X\right) $.

Let $$\alpha \subseteq k\left( Z\right) \setminus k\left( X\right) $$ be a set
of generators of the field $k\left( Z\right) $ over $k\left( X\right) .$ As $%
\alpha $ is finite and $\Delta $ is infinite, there is a finite set $\beta $
such that $$\alpha \subsetneqq \beta \subsetneqq \Delta .$$

By \emph{Step 2}
we have $X_{\beta }\in X_{qc}\left[ \Omega \right] $ such that $X_{\beta }$
is quasi-galois closed over $Z.$
Hence, $X_{qc}\left[ \Omega \right] $ and $%
X_{et}\left[ \Omega \right] $ are cofinal.

Now by \emph{Steps 1-6} above we have
\begin{equation*}
\begin{array}{l}
\pi _{1}^{et}\left( X,s\right)\\

 =\lim_{\leftarrow _{Z\in X_{et}\left[ \Omega %
\right] }}Aut\left( Z/X\right) \\

\cong \lim_{\leftarrow _{Z\in X_{qc}\left[ \Omega \right] }}Aut\left(
Z/X\right) \\

\cong \lim_{\leftarrow _{Z\in X_{qc}\left[ \Omega \right] }}Gal\left(
k\left( Z\right) /k\left( X\right) \right) \\

=\lim_{\leftarrow _{\alpha \in I}}Gal\left( k\left( X_{\alpha }^{un}\right)
/k\left( X\right) \right) \\

\cong Gal\left( k\left( X_{\infty }^{un}\right) /k\left( X\right) \right) \\

=Gal\left( k\left( X\right) ^{un}/k\left( X\right) \right) .
\end{array}
\end{equation*}

This completes the proof.
\end{proof}

\end{document}